\documentclass{amsart}

\usepackage{amssymb}

\newtheorem{theorem}{Theorem}[section]
\newtheorem{lemma}[theorem]{Lemma}

\theoremstyle{definition}

\theoremstyle{remark}

 \newtheorem{corollary}{Corollary}[theorem]

\numberwithin{equation}{section}

\usepackage{amsmath,amsfonts,amsthm,amssymb,amscd,stmaryrd,url,amstext, amsxtra,amsopn , mathtools}
\usepackage{color}
\usepackage{xcolor}
\usepackage{textgreek}
\usepackage{textcomp}

\DeclareFontFamily{OT1}{rsfs}{}
\DeclareFontShape{OT1}{rsfs}{n}{it}{<-> rsfs10}{}
\DeclareMathAlphabet{\mathscr}{OT1}{rsfs}{n}{it}
\DeclarePairedDelimiter\floor{\lfloor}{\rfloor}

 \DeclareMathOperator{\Gal}{Gal}

\DeclarePairedDelimiter\ceil{\lceil}{\rceil}

\newcommand{\Lo}{{\mathscr{L}}}

\renewcommand{\Re}{{\mathfrak{Re}}}
\renewcommand{\Im}{{\mathfrak{Im}}}

\begin{document}

\title{The least prime ideal in the Chebotarev Density Theorem}

\author[H. Kadiri]{Habiba Kadiri}
\address{Department of Mathematics and Computer Science\\
University of Lethbridge\\
4401 University Drive\\
Lethbridge, Alberta\\
T1K 3M4 Canada}
\email{habiba.kadiri@uleth.ca}

\author[N. Ng]{Nathan Ng}
\address{}
\email{nathan.ng@uleth.ca}

\author[P.J. Wong]{Peng-Jie Wong}
\address{}
\email{pengjie.wong@uleth.ca}

\thanks{This research was partially supported by the NSERC Discovery grants of H.K. and N.N.
P.J.W. was supported by a PIMS postdoctoral fellowship and the University of Lethbridge.}

\begin{abstract}
In this article, we prove a new bound for the least prime ideal in the Chebotarev density theorem, which improves the main theorem of Zaman \cite{Z} by a factor of $5/2$. Our main improvement comes from a new version of Tur\'an's power sum method.
The key new idea is to use Harnack's inequality for harmonic functions to derive a superior lower bound for the generalised Fej\'er kernel.
\end{abstract}




\maketitle

\section{Introduction}

Let $L/K$ be a Galois extension of number fields with Galois group $\mathcal{G}$ and 
let  $\mathcal{C} \subset \mathcal{G}$ be a conjugacy class. Attached to each unramified prime ideal $\mathfrak{p}$ in $\mathcal{O}_K$,
the ring of integers of $K$,  is the Artin symbol $\sigma_{\mathfrak{p}}$, a conjugacy class in $\mathcal{G}$. The famous Chebotarev density theorem asserts that  
$$\# \{ \mathfrak{p} \subset \mathcal{O}_K \ | \ N \mathfrak{p} \le x, \ \sigma_{\mathfrak{p}}=\mathcal{C}  \} \sim  
  \frac{|\mathcal{C}|}{|\mathcal{G}|} \text{Li}(x)
$$
as $x \to \infty$, where $ \text{Li}(x)$ is the usual logarithmic integral and
 $N=N_{\mathbb{Q}}^{K}$ is the absolute norm of $K$.
 This is a vast generalisation of both the prime number theorem and Dirichlet's theorem on primes in arithmetic progressions. 

In light of Linnik's celebrated result on the least prime in an arithmetic progression, one may ask for a bound for the least prime ideal whose Artin symbol equals $\mathcal{C}$. Assuming the Generalized Riemann Hypothesis for the Dedekind zeta function of $L$, Lagarias and Odlyzko 
  \cite{LO} showed that  $N \mathfrak{p} \ll (\log d_L)^2$,
 where  $d_L =|\text{disc}(L/\mathbb{Q})|$ is the absolute discriminant of $L$ (see also Bach and Sorenson \cite{BS}). 
 The first unconditional result is due to  
Lagarias, Montgomery, and Odlyzko \cite{LMO}, who proved there exists a positive constant $B$ such that there 
is a prime ideal $\mathfrak{p}$ with $\sigma_{\mathfrak{p}}=\mathcal{C}$ and
$N \mathfrak{p} \le d_{L}^{B}$ for $L \ne \mathbb{Q}$. Recently, Ahn and Kwon \cite{AK} established $B=12577$ 
is valid and 
Zaman \cite{Z} established that $B=40$ is valid when $d_L$ is sufficiently large.

The main result of this article is the following theorem. 
\begin{theorem}  \label{mainthm}
Let $L/K$ be a Galois extension of number fields with $\mathcal{G}=\Gal(L/K)$ and 
let  $\mathcal{C} \subset \mathcal{G}$ be a conjugacy class.  There exists an unramified prime ideal $\mathfrak{p}$ of $\mathcal{O}_K$, of degree one, 
such that $\sigma_{\mathfrak{p}}=\mathcal{C}$ and $N \mathfrak{p} \le d_{L}^{16}$
when $d_L$ is sufficiently large.
\end{theorem}


Our main improvement is a stronger zero-repulsion theorem than in \cite{Z}. We now state our result. Let $\zeta_L(s)$ denote the Dedekind zeta function of $L$. It was shown in \cite{K} that for $d_L$ sufficiently large, 
$\zeta_L(s)$ has no zeros in the region 
\begin{equation}
  \label{zfr}
  \Re(s) \ge 1 - \frac{1}{12.74 \log d_L}  \text{ and } 
  |\Im(s)| \le 1
\end{equation}
with the exception of  at most one real zero $\beta_1$.  The
Deuring-Heilbronn phenomenon roughly asserts that if an {\it exceptional zero} $\beta_1$ exists then 
the zero-free region for $\zeta_L(s)$ can be enlarged. Indeed, we have the following explicit version of this phenomenon.
\begin{theorem}  \label{dh}
Assume $\zeta_L(s)$ admits an exceptional real zero $\beta_1$.   If 
$\beta' +i \gamma'$ is another zero of $\zeta_L(s)$ satisfying $\beta' \ge \frac{1}{2}$ and $|\gamma'| \le 1$, then
for $d_L$ sufficiently large, 
\begin{equation}
     \beta' \le  1 - \frac{ \log \Big( \frac{\kappa}{(1-\beta_1) \log d_L} \Big)}{14.144 \log d_L}
\end{equation}
for some absolute $\kappa >0$.  
\end{theorem}
We remark that Zaman \cite{Z} established the above result with $35.8$ in place of our $14.144$, and that a superior Deuring-Heilbronn phenomenon was obtained by the first two authors for zeros much closer to $s=1$ (see \cite[Theorem 4]{KN}). 
The proof of Theorem \ref{dh} is a special case of the more general Theorem \ref{DH_T} below. 

Also, we have a variant of the above theorem for real zeros. 
\begin{theorem} \label{dhreal}
If $\zeta_L(s)$ has a real zero $\beta_1$ and another real zero 
$\beta'\in(0,1)$, then there is an absolute constant $\kappa >0$ such that for $d_L$ sufficiently large, 
\begin{equation}
     \beta' \le  1 - \frac{ \log \Big( \frac{\kappa}{(1-\beta_1) \log d_L} \Big)}{7.071998 \log d_L}.
\end{equation}
\end{theorem}
The  proof of Theorem \ref{dhreal} is a special case of the more general Theorem \ref{DH_0} below. 
This gives an upper bound for the exceptional zero $\beta_1$ of $\zeta_L(s)$. 
\begin{corollary}  \label{upperbd}
For $d_L$ sufficiently large, if $\zeta_L(s)$ has a real zero $\beta_1$,  then
\[
   1-\beta_1 \gg d_{L}^{-7.072},
\]
where the implied constant is absolute and effectively computable. 
\end{corollary}
Note that Theorem \ref{dhreal} and Corollary \ref{upperbd} improve results of Zaman, who obtained the above bounds with $16.6$ in place of our $7.071998$ and $7.072$.   The proof of Corollary \ref{upperbd} follows from applying Theorem \ref{dhreal}
to the zero $\beta'=1-\beta_1$.  

\noindent {\it Remarks}.   We actually obtain the bound $d_{L}^{B}$ with $B=15.72$ in Theorem \ref{mainthm} and this likely can be slightly improved.  It seems that limit of our method is roughly the constant $14.144$  which appears in the Deuring-Heilbronn 
type zero-free region as given in Theorem \ref{dh}.
More precisely, as can be seen in \eqref{firstassumptions}, the exponent 15.72  arises from the value $C_1=14.58$ in Theorem \ref{DH_T} (see equation \eqref{highzfr}). 
 In fact, our constant $B$ is chosen to satisfy the constraints \eqref{firstassumptions} and \eqref{conditionsBT} below. 
Finally, it should be mentioned that the work of Thorner and Zaman \cite{TZ} gives a significant improvement to Theorem \ref{mainthm} in the case that there is a large abelian subgroup of $G$ intersecting $C$.

\section{Deuring-Heilbronn via the Tur\'an Power Sum method}

\subsection{A new version of Tur\'an's power sum method} 
In this section, we shall prove a version of Tur\'an's power sum method in much the same spirit as Lagarias et al. \cite{LMO}, who made use of properties of functions of the form 
\[ 
  P(r,\theta) = \sum_{j=1}^{J} \Big( 1-\frac{j}{J+1} \Big) r^j \cos(j \theta),
\]
where $0 \le r \le 1$.  We have the following facts concerning $P(r,\theta)$. 
\begin{lemma} \label{fejer}
Let $\theta \in \mathbb{R}$. We have
\begin{itemize}
 \item[(i)] $P(r,\theta)  \ge - \frac{1}{2} \text{ for }  0 \le r \le 1$,
 \item[(ii)]  $P(1,0)  = \frac{J}{2}$,
 \item[(iii)]  $P(r,\theta)  \ge -\frac{r}{1+r}   \text{ for }  0 \le r  \le 1$.
\end{itemize}
\end{lemma}
\begin{proof}
(i) Note that $1+2 P(1,\theta)$ is Fej\'er's kernel and thus non-negative.   It follows that, $P(1,\theta) \ge -\frac{1}{2}$.  Since $P(r,\theta)$ is a harmonic function,
we have $P(r,\theta) \ge - \frac{1}{2}$ for $0 \le r < 1$. 
(ii) is a direct calculation. 
To prove (iii), we first note that the case $r=1$ follows from (i), so we may assume that $0 \le r < 1$. We then
observe 
that by writing $(x,y)=(r\cos\theta,r\sin\theta)$, $F(x,y)=1+2P(r,\theta)$ is the real part of the kernel
$
1+2\sum_{j=1}^{J}( 1-\tfrac{j}{J+1} )  r^j e^{ij\theta}$.
Thus, $F$ is a non-negative harmonic function defined on the unit ball, centred at $(0,0)$, of $\Bbb{R}^2$. Recall that Harnack's inequality (see e.g. \cite[Theorem 2.1]{Ka07}) asserts that for any non-negative harmonic function $H$ defined over an open ball $\|(x,y)\|_{\Bbb{R}^2}<R$, then for any $\|(x,y)\|_{\Bbb{R}^2}=r$ with $r<R$,
$$
\frac{R-r}{R+r}H(0,0)\le H(x,y).
$$
Applying this with $H=F$ and $R=1$, we have  
$
\tfrac{1-r}{1+r}F(0,0) \le F(x,y)
$
for any $\|(x,y)\|_{\Bbb{R}^2}= r<1.$ As $F(0,0)=1$, we then deduce that for any $r<1$,
$
\tfrac{1-r}{1+r}-1\le 2P(r,\theta)
$,
which together with a simplification completes the proof.
\end{proof}
\noindent {\it Remarks}. (i) The authors of \cite{LMO}  and \cite{Z} use  the lower bound 
$P(r,\theta) \ge -\min(\frac{3r}{2},\frac{1}{2})$ for $0 \le r \le 1$ to  obtain a version of Tur\'an's power sum method. 
The main improvements obtained in Theorems \ref{mainthm}, \ref{dh},  \ref{dhreal}, and \ref{newTuran} come from our use of  Lemma \ref{fejer} (iii). \\
(ii)  As the proof of Theorem \ref{newTuran} relies on Lemma \ref{fejer} (ii), one may try to replace $P(r,\theta)$ by a function of the form  $
  \widetilde{P}(r,\theta) = \sum_{j=1}^{J} a_j r^j \cos(j \theta)$ which satisfies the first condition of Lemma \ref{fejer} and $\widetilde{P}(1,0)>\frac{J}{2}$.
 This leads  to  the extremal problem of finding $\max \{a_1+\ldots +a_J \mid \sum_{j=1}^J a_j  \cos(j \theta)\ge -\frac{1}{2} \}$.
However, $P(1,\theta)$ is, in fact, \emph{the} solution of such a problem (see, for example, \cite[Section 4]{Di}). 

Now we are in a position to prove a new version of Tur\'an's power sum method. This  refines \cite[Theorem 4.2]{LMO}, and will help us to enlarge the zero-free region later.
\begin{theorem}\label{newTuran}
Let $\varepsilon>0$.
For any $j\in \Bbb{N}$, set $s_j=\sum_{n\ge 1}b_nz_n^j$, and assume that
\begin{itemize}
 \item[(i)]  $z_n \in \mathbb{C}$ and $|z_n|\le |z_1|$ for every $n\ge 1$ where $z_1 \ne 0$, 
 \item[(ii)] each $b_n$ is non-negative and $b_1 >0$.
\end{itemize}
Set
$
M=b_1^{-1}\sum_{n\ge 1}\frac{b_n |z_n|}{|z_1|+|z_n|}$.
Then there exists $j$ with $1\le j\le (8+\varepsilon)M$ such that 
\[
\Re(s_{j})
\ge  \frac{ b_1 \varepsilon }{32+4\varepsilon} |z_1|^{j}.
\]
\end{theorem}
\begin{proof}
Let $\varepsilon >0$.
By homogeneity, we may assume $|z_1|=1$. Write $z_{n} = r_{n} e^{i \theta_{n}}$ where $0 \le r_n \le 1$.
Note that $r_1=1$. 
Observe that we have the inequality
\begin{equation}
  \label{maxsj}
\sum_{j=1}^{J} \Big( 1 - \frac{j}{J+1} \Big) \Re(s_{j})(1+\cos(j \theta_1))\le J \Big( \max_{1 \le j \le J} \Re(s_j)  \Big).
\end{equation}
By definition, the sum on the left of \eqref{maxsj} equals
\begin{equation}\label{sumP}
   \sum_{n \ge 1}b_n \Big( P(r_{n },\theta_{n})
    + \frac{1}{2}  P(r_{n },\theta_{n }-\theta_1)
      + \frac{1}{2}  P(r_{n },\theta_{n }+\theta_1)
   \Big).
\end{equation}   
By  Lemma \ref{fejer} (ii) and (iii), 
a lower bound for the $n=1$ term is 
\begin{equation}
  \label{n1}
b_1 \Big( - \frac{r_1}{1+r_1} + \frac{J}{4} - \frac12 \frac{r_1}{1+r_1}\Big)
      =
b_1 \Big( \frac{J+1}4
      - 2 \frac{r_1}{1+r_1}\Big).
\end{equation}
For the  terms $n \ge 2$, Lemma \ref{fejer} (iii) bounds the remaining sum below by 
\begin{equation}
  \label{ngreater2}
  \sum_{n=2}^{\infty} -b_n \frac{r_n}{1+r_n} \Big( 1+ \frac{1}{2}+ \frac{1}{2} \Big) =
 -2  \sum_{n=2}^{\infty} \frac{b_n r_{n }}{1+r_{n }}.
\end{equation}
Combining \eqref{n1} and \eqref{ngreater2} it follows that  a lower bound for  \eqref{sumP} is 
\[
  b_1 \left(  \frac{J+1}{4} -2 M\right)
\ge  b_1 \frac{ \varepsilon M }{4} 
\text{ for } J= \floor*{(8+\varepsilon)M}.
\]
In particular, there exists $1\le j \le (8+\varepsilon)M$ such that 
$
\Re(s_{j})\ge \frac{b_1 \varepsilon M }{4J} 
\ge  
 \frac{b_1 \varepsilon }{4 (8+\varepsilon)}
$.
\end{proof}
\noindent {\it Remarks}. (i)
Here we only apply this with $b_n \equiv 1$, which was also the case for \cite{Z}.
In addition, thanks to Lemma \ref{fejer} we reduce to $8$ and respectively to $32$ the constants $12$ and $48$ appearing in \cite[Theorem 2.3]{Z}.  \\
(ii) This theorem can be used to improve the main theorem in \cite{AK} and some of the theorems in 
\cite{TZ}.  This is currently work in progress. 

\subsection{A quantitative version of Deuring-Heilbronn phenomenon} 

In this section, we will employ our version of Tur\'an's power sum method to prove a quantitative version of the Deuring-Heilbronn phenomenon.  
We require a result of Odlyzko on discriminants of number fields. (We shall note that the bound \eqref{dLbound} below is, in fact, a weak form of \cite[Theorem 1, Eq. (4)]{O}.) 
 \begin{theorem}[Odlyzko, \cite{O}]   \label{Odlyzko}
 Let $L$ be a number field of degree $n_L$ and discriminant $d_L$. Let $r_1=r_1(L)$ and $2r_2=2r_2(L) $ denote the number of real and complex embeddings of $L$, respectively. 
 For $d_L$ sufficiently large, 
 \begin{equation}
   \label{dLbound}
     (\log 60) \cdot r_1  + (\log 22) \cdot 2r_2 \le \log d_L. 
 \end{equation}
\end{theorem}
We shall also borrow a lemma from Zaman, which bounds certain sums of zeros of Dedekind zeta functions. For $\alpha \ge 1$ and $t \in \mathbb{R}$,
we define
\begin{equation}\label{def-mathscr{S}}
   \mathscr{S}_L(\alpha,t)
   = \sum_{\rho}  \Big( \frac{1}{|\alpha+1-\rho|^2}
   + \frac{1}{|\alpha+1+it-\rho|^2 } \Big),
\end{equation}
where $\rho$ ranges through non-trivial zeros of $\zeta_L(s)$. 
We further define the functions
\begin{align*}
 & G_1(\alpha;t) := \frac{\Delta(\alpha+1,0)+\Delta(\alpha+1,t)}{2}-\log \pi,  \\
 & G_2(\alpha;t)  := \frac{\Delta(\alpha+1,0)+\Delta(\alpha+2,0)+\Delta(\alpha+1,t)+\Delta(\alpha+2,t) }{4}-\log \pi,  
\\
 & \text{and }\Delta(x,y)  := \Re \Big( \frac{\Gamma'}{\Gamma} \Big(\frac{x+iy}{2} \Big)  \Big).
\end{align*}
In \cite[Lemma 2.5 and Eq. (2.10)]{Z}, Zaman showed the following estimate.
\begin{lemma} \label{lemma-explicit-ineq-Zaman} Let $L$ be a number field with $r_1$ and $2r_2$,
the number of real and complex embeddings, respectively.    Let $\Lo = \log d_L$. 
For $\alpha \ge 1$ and $t \in \mathbb{R}$, we have
\begin{equation}  \label{Zamanexplicit}
    \mathscr{S}_L(\alpha,t)
  \le \frac{1}{\alpha} \Lo +  \frac{G_1(\alpha;|t|)}{\alpha} \cdot r_1
    + \frac{G_2(\alpha;|t|)}{\alpha}  \cdot 2r_2 + \frac{2}{\alpha^2} + \frac{2}{\alpha +\alpha^2}.
\end{equation}
\end{lemma}
We now use Theorem \ref{newTuran} to derive an inequality for a power sum associated to the exceptional zero $\beta_1$ of $\zeta_L(s)$. Let the set of non-trivial zeros (resp., trivial zeros) of $\zeta_L(s)$ be denoted by $\mathcal{S}$ (resp., $\mathcal{T}$), and
set $\mathcal{S}(T) = \{ \rho \in \mathcal{S} \ | \ |\Im(\rho)| \le T \}.$ For $j \in \mathbb{N}$ and $c \ge 2$, define 
\begin{equation}\label{def-Sj}
   S_j =  \sum_{\rho \in \mathcal{S} \backslash \{ \beta_1 \}}  \frac{1}{(c-\rho)^{2j}}
   + \sum_{\rho \in \mathcal{S}  \backslash \{ \beta_1 \}} \frac{1}{(c+it-\rho)^{2j}}.
\end{equation}
\begin{lemma} \label{turanreplacement}
Let $c\ge 2$, $T\ge 1$, $0 < \eta < 1$, and $0<\varepsilon <1$. 
Let $\beta_1$ be an exceptional zero satisfying $\beta_1 > 1-\frac{1}{12.74 \Lo}$ and let $\rho'=\beta'+it \in \mathcal{S}(T) \backslash \{ \beta_1 \}$
be such that $\beta'$ is maximal and $|t|$ is chosen minimally. If $1-\eta \le \beta' < 1$, then for $d_L$  sufficiently large
there exists  $a >0$ and $\delta >0$  such that 
\begin{equation}
  \label{ReSjlowerbd}
   \Re(S_j) \ge  \delta  (c-\beta')^{-2j}
\end{equation}
for some $1 \le j \le (8+ 2\varepsilon)(c-1+\eta)^2 a  \Lo$, where 
\begin{equation}
      \label{def-a}
      a := a(c,\eta,T) = \frac{1}{d-1} +  \max_{|t|\le T} 
         \Bigg\{ \frac{G_1(d-1;|t|)}{ (d-1) \log 60},
    \frac{G_2(d-1;|t|)}{(d-1) \log 22},0  \Bigg\}
\end{equation}
and 
\begin{equation}
  \label{d-def}
  d := d(c,\eta) =  \sqrt{2c^2+(1-\eta)^2-2c(1-\eta)}.
\end{equation} 
\end{lemma}
\begin{proof}
Let $0<\varepsilon <1$ be a positive parameter. 
Throughout this proof a non-trivial zero of $\zeta_L(s)$ is denoted $\rho=\beta+i \gamma$.
Without loss of generality, we may assume $t \ge 0$. Let 
$m = \min  \{  |c-\rho|,    |c+it -\rho| \ | \ \rho \in \mathcal{S} \} $.
By definition we have that $m \le |c+it-\rho'| = c-\beta'$.
Since we assume $\beta' \ge 1-\eta$, it follows that 
$m \le c-(1-\eta)=c-1 +\eta$.
It shall be convenient to define  
\begin{equation}\label{def-A}
  \mathcal{A} := \mathcal{A}(c,\eta)= (c-1+\eta)^2. 
\end{equation}
Observe that we have $m^2 \le \mathcal{A}$ and also $\mathcal{A} \ge (c-1)^2$. 
Now write 
\[ z_{\rho} = \frac{m^{2}}{(c-\rho)^{2}}, \ 
  w_{\rho} = \frac{m^{2}}{(c+it-\rho)^{2}}, \ \text{and }\  
   S_{j}' = 
  \sum_{\rho \in \mathcal{S} \backslash \{ \beta_1 \}} z_{\rho}^{j}
  +   \sum_{\rho \in \mathcal{S} \backslash \{ \beta_1 \}} w_{\rho}^{j}.
\]
Observe that 
\begin{equation}
  \label{SjSjprime}
S_j = m^{-2j} S_{j}'  \ge (c-\beta')^{-2j} S_{j}' 
\end{equation}
where we recall that $S_j$ is defined by \eqref{def-Sj}. 
The normalization guarantees that $|z_{\rho}|,|w_{\rho}| \le 1$. 
Let  $z_{\rho} = r_{\rho} e^{i \theta_{\rho}}$ and $w_{\rho} = s_{\rho} e^{i \alpha_{\rho}}$.
We then apply Theorem \ref{newTuran}, in the special case $b_n=1$,  to the multiset $\{ z_{\rho},w_{\rho}: \rho\not=\beta_1\}$.
In this case we have 
\[
M =   \sum_{\rho \in \mathcal{S} \backslash \{ \beta_1 \}} \frac{r_{\rho}}{1+r_{\rho}}
  +   \sum_{\rho \in \mathcal{S} \backslash \{ \beta_1 \}} \frac{s_{\rho}}{1+s_{\rho}},
  \ J= \floor*{(8+\varepsilon)M},
\]
and there exists $j$ satisfying $1\le j \le J$ such that 
\begin{equation}
  \label{eq:sigmaub}
\Re(S_{j}'  ) \ge \frac{\varepsilon}{32+4\varepsilon}.
\end{equation}
We now provide an upper bound for $J$.  To do this, we must bound $M$. 
As 
\[
   r_{\rho}  \le \frac{\mathcal{A}}{(c-\beta)^2 + \gamma^2}\ \text{ and } \ 
     s_{\rho}  \le   \frac{\mathcal{A}}{(c-\beta)^2+(\gamma-t)^2 }, 
\]
it follows that
\begin{equation}
  \label{rsinequalities}
\frac{ r_{\rho} }{1+ r_{\rho} } 
  \le  
  \frac{\mathcal{A}}{(c-\beta)^2+\gamma^2 +\mathcal{A}}
\ \text{ and } \ 
 \frac{ s_{\rho} }{1+ s_{\rho} } 
   \le  \frac{\mathcal{A}}{(c-\beta)^2+(\gamma-t)^2+\mathcal{A}} .
\end{equation}
Thus 
\[
M \le  \mathcal{A}  \sum_{\rho \in \mathcal{S} \backslash \{ \beta_1 \}}
\left(
 \frac{1}{(c-\beta)^2+\gamma^2 +\mathcal{A}}
+  \frac{1}{(c-\beta)^2+(\gamma-t)^2+\mathcal{A}} 
\right).
\]
By \eqref{def-A} and the fact that $c\ge 2$, 
it may be checked that for $\beta \in (0,1)$, 
\begin{equation}
  \label{cdinequality}
  (c-\beta)^2 +\mathcal{A}  =  (c-\beta)^2 + (c-1+\eta)^2 \ge (d-\beta)^2 , 
\end{equation}
where $
  d := d(c,\eta)= \sqrt{c^2+\mathcal{A}} = \sqrt{2c^2+(1-\eta)^2-2c(1-\eta)}$
  agrees with \eqref{d-def}. 
Thus $
M \le  \mathcal{A}  \mathscr{S}_L(d-1,t)$,
where we recall $ \mathscr{S}_L$ is defined in \eqref{def-mathscr{S}}. 
We are now able to apply the explicit inequality from Lemma \ref{lemma-explicit-ineq-Zaman} to bound $M$: 
$$
M \le  \mathcal{A} \Big(
\frac{\Lo}{d-1}  +  \frac{G_1(d-1;|t|)}{d-1} \cdot r_1
    + \frac{G_2(d-1;|t|)}{d-1}  \cdot 2r_2 + \frac{2}{(d-1)^2} + \frac{2}{d-1 +(d-1)^2}
\Big).
$$
Now let 
\begin{equation}\label{Mdef}
\mathfrak{M}(T;\eta)=   \max_{|t|\le T} 
         \Bigg\{ \frac{G_1(d-1;|t|)}{ (d-1) \log 60},
    \frac{G_2(d-1;|t|)}{(d-1) \log 22},0  \Bigg\}. 
\end{equation}
For $d_L$ sufficiently large, we have the 
 bound \eqref{dLbound} of 
Theorem \ref{Odlyzko} and thus 
\begin{equation}
\label{bound-M1}
M \le  \mathcal{A} \Bigg( \frac{1}{d-1} 
 + 
 \mathfrak{M}(T;\eta) 
+ \frac{2}{(d-1)^2 \Lo} + \frac{2}{(d-1 +(d-1)^2) \Lo} \Bigg) \Lo.
\end{equation}
Thus if $0 < \varepsilon < 1$ and $d_L$ is sufficiently large in terms of $\varepsilon,\eta$, and $c$, 
%
$
  M \le  
  \Lo \mathcal{A} a(1+  \tfrac{\varepsilon}{9})$.
It follows that there exists $j$ satisfying $1\le j \le (8+\varepsilon) M \le (8+\varepsilon)(1+ \tfrac{\varepsilon}{9}) \mathcal{A}a \Lo
\le (8+2 \varepsilon) \mathcal{A}a \Lo$ such that \eqref{eq:sigmaub} holds. 
Combining this with \eqref{SjSjprime} we obtain \eqref{ReSjlowerbd} and we complete the proof. 
%
\end{proof}
From the last lemma, we deduce the following version of the Deuring-Heilbronn phenomenon. 
\begin{theorem}\label{DH_T}
Let  $c \ge 2$, $\eta \in(0,1)$, $T \ge 1$, and $0 <\varepsilon < 1$. 
Suppose that $\beta_1 = 1- \frac{\lambda_1}{\Lo}$ is an exceptional zero of $\zeta_L(s)$ and $\rho'= \beta'+i \gamma'$ is another zero of $\zeta_L(s)$ with $|\gamma'| \le T$. Write $\beta'=1- \frac{\lambda'}{\Lo}$. Then either $\beta' \le 1-\eta$ or there exist positive effectively computable constants 
$C$ and $\kappa$ such that 
 \begin{equation}
   \label{lambdaprimeC}
  \lambda' \ge \frac{1}{C} \log(\kappa \lambda_{1}^{-1}), \ 
  C = C(c,\eta,T,\varepsilon)= \frac{\mathcal{A} a (16+4\varepsilon)}{ c-1},
\end{equation}
where $a$ and $A$ are defined in \eqref{def-a} and \eqref{def-A}. 
\end{theorem}
From this theorem we obtain the following values for $C=C(c,\eta,T,\varepsilon)$.
\begin{equation}
  \label{lowzfr}
   C(5,0.5,1, \tilde{\varepsilon})= 14.1439 \ldots, \ C(5.2,0.1,1, \tilde{\varepsilon})=12.2618 \ldots,
\end{equation}
\begin{equation}
  \label{highzfr}
  C(4.6,\tfrac{1}{2},4.6, \tilde{\varepsilon}) = 14.58 \ldots, 
   C(5.8,\tfrac{1}{2},10,\tilde{\varepsilon})= 15.50 \ldots,  
 C(10.3,\tfrac{1}{2},130,\tilde{\varepsilon})=20.21 \ldots
\end{equation}
where $\tilde{\varepsilon}=10^{-11}$.
Observe that Theorem \ref{dh} follows from the first entry of \eqref{lowzfr}. 
\begin{proof}
As explained in \cite[p.289]{LMO}, there are $\alpha$ and $r$, depending on $L$, such that 
\[
     -\frac{\zeta_{L}'}{\zeta_L}(s) = \frac{1}{s-1}-\alpha - \sum_{\omega  \in \mathcal{S} \cup \mathcal{T} \backslash \{ 0 \} }  \Big( \frac{1}{s-\omega} + \frac{1}{\omega} \Big) - \frac{r}{s} = \sum_{\mathfrak{p}} \sum_{m=1}^{\infty} (\log N \mathfrak{p}) (N \mathfrak{p})^{-ms}
\]
for $\Re(s) >1$.  
Differentiating the above equation $2j-1$ times yields 
\begin{equation}
  \label{diffexplicit0}
 \frac1{(2j-1)!}  \sum_{\mathfrak{p}} \sum_{m=1}^{\infty}  (\log N \mathfrak{p}) (\log N \mathfrak{p}^m)^{2j-1} (N \mathfrak{p})^{-ms}= \frac{1}{(s-1)^{2j}}- \sum_{\omega \in \mathcal{S} \cup \mathcal{T}} \frac{1}{(s-\omega)^{2j}}.
\end{equation}
For $c > 1$, we apply this identity with $s=c$ and $s=c +i\gamma'$ to obtain 
\begin{equation}
\begin{split}
  \label{diffexplicit}
   & \frac{1}{(2j-1)!}  \sum_{\mathfrak{p}} \sum_{m=1}^{\infty}  (\log N \mathfrak{p}) (\log N \mathfrak{p}^m)^{2j-1} (N \mathfrak{p})^{-mc }
   (1+ (N \mathfrak{p}^m)^{-i\gamma'})  \\
   & =  \frac{1}{(c -1)^{2j}} + \frac{1}{(c+i\gamma'-1)^{2j}} - \frac{1}{(c -\beta_1)^{2j}}- \frac{1}{(c+i\gamma'-\beta_1)^{2j}}
   \\&  - \sum_{\omega \in \mathcal{S} \cup \mathcal{T} \backslash \{ \beta_1 \}} 
   \left( \frac{1}{(c -\omega)^{2j}}
   + \frac{1}{(c+i\gamma'-\omega)^{2j}} \right) .
\end{split}
\end{equation}
As $1+ \cos(\gamma'm \log N \mathfrak{p}) \ge 0$, the real part on the left is positive, which gives
\begin{multline*}
 \frac1{(c-1)^{2j}} +  \Re \frac{1}{(c+i\gamma'-1)^{2j}} - \frac{1}{(c-\beta_1)^{2j}}-  \Re \frac{1}{(c+i\gamma'-\beta_1)^{2j}}
 \\ - \sum_{\omega \in \mathcal{S} \cup \mathcal{T} \backslash \{ \beta_1 \}} 
   \Re \left( \frac{1}{(c -\omega)^{2j}}+ \frac{1}{(c+i\gamma'-\omega)^{2j}} \right)
    \ge 0. 
\end{multline*}
Observe that for $t' \in  \mathbb{R}$, we have the bound 
\begin{equation}
    \left| \frac{1}{(c+i t'-1)^{2j}} - \frac{1}{(c+it'-\beta_1)^{2j}} \right|  = 2j \left| \int_{\beta_1}^{1} (c+it'-x)^{-2j-1} dx \right| 
\le \frac{2j(1-\beta_1)}{(c-1)^{2j+1}}\label{bdd}.
\end{equation}
%
Applying this with $t'=0$ and $t'=\gamma'$, we derive
\[
    \frac{ 4 j(1-\beta_1)}{(c-1)^{2j+1}}
    \ge \sum_{\omega \in \mathcal{S} \cup \mathcal{T}  \backslash \{ \beta_1 \}}  \Re \left( \frac{1}{(c-\omega)^{2j}} +
\frac{1}{(c+i\gamma'-\omega)^{2j}} \right). 
\]
Since $-\Re(z) \le |z|$, for any $\omega\in\Bbb{R}$, one has 
\begin{equation}\label{trivialzeros}
  -\Re \frac{1}{(c+i\gamma'-\omega)^{2j}}  \le \frac{1}{|c+i\gamma'-\omega|^{2j}}
  = \frac{1}{((c-\omega)^2 + \gamma'^2)^j} \le \Re \frac{1}{(c-\omega)^{2j}}.
\end{equation}
It follows that for the trivial zeros we have
\[
    \sum_{\omega \in  \mathcal{T}  }  \Re \left( \frac{1}{(c-\omega)^{2j}}
   + \frac{1}{(c+i\gamma'-\omega)^{2j}} \right)
   \ge 0.
\]
By Lemma \ref{turanreplacement}, there exist $a>0, \delta>0$, and $1 \le j \le (8+2\varepsilon)\mathcal{A} a \Lo$ so that
\[
    \sum_{\omega \in \mathcal{S} \backslash \{ \beta_1 \}}  \Re
    \left(  \frac{1}{(c-\omega)^{2j}} + \frac{1}{(c+i\gamma'-\omega)^{2j}} \right)
   \ge \delta (c-\beta')^{-2j}. 
\]
Combining these estimates, we then deduce 
$\tfrac{4 j(1-\beta_1)}{(c-1)^{2j+1}}  \ge \delta (c-\beta')^{-2j}$  
and thus
\[
  \frac{4j(1-\beta_1)}{c-1} \ge \delta  \Big( \frac{c-1}{c-\beta'} \Big)^{2j}
  = \delta \exp \Big( 
  -2j \log \Big( \frac{c-\beta'}{c-1} \Big)
  \Big). 
\]
Observe that 
\[
      \log \Big( \frac{c-\beta'}{c-1} \Big)
  = \log \Big( 
  1 + \frac{1-\beta'}{c-1}
  \Big) \le \frac{1-\beta'}{c-1}
\]
since $\log(1+x) \le x$ for $x \ge 0$. 
As we have $j \le (8+2\varepsilon) \mathcal{A} a \Lo$, it follows that 
\[
   \frac{4\cdot (8+2\varepsilon) \mathcal{A} a \Lo (1-\beta_1)}{c-1} \ge \delta  \exp\left(-\frac{2\cdot (8+2\varepsilon) \mathcal{A}a\Lo(1-\beta')}{c-1}\right).
\]
Recalling $\beta_1 =  1-\frac{\lambda_1}{\Lo}$ and $\beta'= 1-\frac{\lambda'}{\Lo}$, we see this is equivalent to 
\[
(32+8\varepsilon) \mathcal{A} a  \lambda_1 \ge \delta (c-1)  \exp\left(-\frac{(16+4\varepsilon) \mathcal{A}a  \lambda' }{c-1}\right).
   \]
Rearranging this implies the inequality in 
 \eqref{lambdaprimeC}.
\end{proof}
%
\subsection{Location of other real zeros}
In this section, we derive a version of the Deuring-Heilbronn phenomenon that \emph{only} concerns real zeros of $\zeta_L(s)$. To do this, we shall keep using the notation introduced in the previous sections, but we now instead consider, for fixed $c \ge 2$ and $j \in \mathbb{N}$, 
\[
   \tilde{S}_j =  \sum_{\rho \in \mathcal{S} \backslash \{ \beta_1 \}}  \frac{1}{(c-\rho)^{2j}},
\]
where $\mathcal{S}$ again stands for the set of non-trivial zeros of $\zeta_L(s)$. As in the previous section, we first require the following version of power sum inequality for $\tilde{S}_j$'s. 
\begin{lemma} \label{turanreplacement_t=0}
Let $0 < \eta < 1$ and $0 < \varepsilon < 1$. 
Let $\beta_1$ be an exceptional zero satisfying $\beta_1 > 1-\frac{1}{12.74 \Lo}$ and let $\rho'=\beta'$ be another real zero with maximal $\beta'$.  If $1-\eta \le \beta' < 1$, then 
there exist  $a' >0$ and $\delta'  >0$  such that 
\[
   \Re(\tilde{S}_j) \ge  \delta ' (c-\beta')^{-2j}  
\]
for some $1 \le j \le 4(c-1+\eta)^2 a'(1+2\varepsilon)  \Lo$
where 
\begin{equation}\label{adefn'}
  a' :=a'(c,\eta)=          
       \frac{1}{d-1} 
       + \mathfrak{M}(0;\eta)
\end{equation}
and we recall that $d$ is defined in \eqref{d-def} and $\mathfrak{M}(T;\eta)$ is defined in \eqref{Mdef}.
%
\end{lemma}

\begin{proof}
As the proof is almost the same as the proof of Lemma \ref{turanreplacement} with $t=T=0$, we shall just give a sketch of the proof and outline some differences here. Now we let 
$m = \min_{\rho \in \mathcal{S}} |c-\rho|$.
Clearly, we have that $m \le c-\beta'$.
As $\beta' \ge 1-\eta$, it again follows that 
$m^2 \le \mathcal{A}:= (c-1+\eta)^2$. 
Now write 
\[ 
   \tilde{S}_j = m^{-2j}\tilde{S}_{j}' = 
         m^{-2j}\sum_{\rho \in \mathcal{S} \backslash \{ \beta_1 \}} z_{\rho}^{j},
\]
and $z_{\rho} = \frac{m^{2}}{(c-\rho)^{2}}$. 
By following the same argument as in Theorem \ref{newTuran} we establish 
\begin{equation}
  \label{eq:sigmaub'}
    J\Big( \max_{1 \le j \le J} \Re(\tilde{S}_j')  \Big)\ge\frac{J+1}{4} -2M,
\end{equation}
where $M = \sum_{\rho \in \mathcal{S} \backslash \{ \beta_1 \}  }\frac{|z_{\rho}|}{1+|z_{\rho}|} $.  
By \eqref{rsinequalities} and \eqref{cdinequality} we have 
\[
    \frac{|z_{\rho}|}{1+|z_{\rho}|} \le  
      \frac{\mathcal{A}}{(c-\beta)^2+\gamma^2 +\mathcal{A}} \le \frac{\mathcal{A}}{(d-\beta)^2 + \gamma^2}
\]
where $d$ and $\mathcal{A}$ are given by \eqref{d-def} and \eqref{def-A}. 
Therefore $2M \le \mathcal{A}  \mathscr{S}_L(d-1,0)$ where $\mathscr{S}_L$ is given by \eqref{def-mathscr{S}}. 
Furthemore, by \eqref{Zamanexplicit}, $\mathscr{S}_L(d-1,0) \le a'(1 + \varepsilon) \Lo$ for $d_L$ sufficiently large.
Combining the last two inequalities with  \eqref{eq:sigmaub'} we have 
\[
  J \max_{1 \le j \le J} \Re(\tilde{S}_j') \ge 
   \frac{J+1}{4} - \mathcal{A} a'(1 +\varepsilon) \Lo\ge \mathcal{A} a' \varepsilon \Lo
 \]
by the choice $J = \lfloor 4 \mathcal{A}a'(1+2\varepsilon) \Lo \rfloor$.  Therefore, 
it follows that there exists  $1 \le j \le J$ such that $\tilde{S}_{j}' \ge  \frac{\mathcal{A}a' \varepsilon \Lo}{J}
\ge \delta'= \frac{\varepsilon}{4(1+2 \varepsilon)}$ and thus 
 $\tilde{S}_j \ge (c-\beta)^{-2j} \tilde{S}_j' \ge \delta'(c-\beta)^{-2j}$ as desired.
\end{proof}

We now deduce a version of the Deuring-Heilbronn phenomenon for real zeros of $\zeta_L(s)$. 
\begin{theorem}\label{DH_0}
Let  $c \ge 2$, $\eta \in(0,1)$, and $0 < \varepsilon < 1$. 
If $\beta_1 = 1- \frac{\lambda_1}{\Lo}$ is an exceptional zero of $\zeta_L(s)$ and $\beta'=1- \frac{\lambda'}{\Lo}$ is another real zero of $\zeta_L(s)$. Then either $\beta' \le 1-\eta$ or there exist positive effectively computable constants 
$C'$ and $\kappa'$ such that 
 \begin{equation}
   \label{lambdaprime'C'}
  \lambda' \ge \frac{1}{C'} \log(\kappa' \lambda_{1}^{-1}), \ 
  C'= C'(c,\eta,\varepsilon)=\frac{8\mathcal{A}a'(1+4\varepsilon)}{c-1}.
\end{equation}
\end{theorem}
Note that Theorem \ref{dhreal} follows from this since $C'(5,\tfrac{1}{2},\tfrac{1}{10^{11}}) = 7.07199744\ldots$.  
\begin{proof}
By applying \eqref{diffexplicit0} with  $s=c$ and then taking real parts and using non-negativity, we have 
$$
 \frac1{(c-1)^{2j}}  - \frac{1}{(c-\beta_1)^{2j}} - \sum_{\omega \in \mathcal{S}  \cup \mathcal{T} \backslash \{ \beta_1 \}}  \Re \frac{1}{(c-\omega)^{2j}}
   \ge 0. 
$$
This together with the bound \eqref{bdd} and inequality \eqref{trivialzeros} (with $\gamma'=0$) yields 
\[
    \frac{ 2 j(1-\beta_1)}{(c-1)^{2j+1}}
    \ge \sum_{\omega \in \mathcal{S} \backslash \{ \beta_1 \}}  \Re \frac{1}{(c-\omega)^{2j}}, 
\]
which by Lemma \ref{turanreplacement_t=0}, is greater than $\delta' (c-\beta')^{-2j}$ for $a'$ given by \eqref{adefn'}, some $\delta'>0$, and $1 \le j \le 4\mathcal{A} a'(1+2\varepsilon) \Lo$.
Finally, using the exactly same argument as in the last few lines of the proof of Theorem \ref{DH_T}, we then deduce \eqref{lambdaprime'C'}.
\end{proof}

\section{The least prime ideal}
To prove Theorem \ref{mainthm}, we shall use the method based on Heath-Brown's proof of Linnik's theorem  \cite{HB} as adapted to the number field setting by Zaman \cite{Z}.  Let us consider $\mathscr{P} = \{ \mathfrak{p} \subset \mathcal{O}_K \ | \  \mathfrak{p} \text{ is unramified},
 \mathfrak{p} \text{ is of degree 1, and } 
\sigma_{\mathfrak{p}}=\mathcal{C}
  \}$ and
\begin{equation}
  \label{S}
  S = \sum_{\mathfrak{p} \in \mathscr{P}} \frac{\log N \mathfrak{p}  }{N \mathfrak{p} }
  f \Big(  \frac{\log N \mathfrak{p}  }{\Lo } \Big) ,
\end{equation}
where $f$ is a compactly supported function.  

We shall evaluate $S$ with the weights $f=f_{\ell,A,B}(t)$, described in the following lemma.   This family of weights was introduced by Zaman \cite[Lemma 2.6]{Z}, who generalised the weights used by Heath-Brown \cite{HB}.
\begin{lemma}
For $A,B >0$ and $\ell\in \Bbb{Z}^+$ satisfying $B>2A \ell$,
there exists a real variable function $f(t) = f_{\ell,A,B}(t)$ satisfying the following properties. 
\begin{itemize}
\item[(i)] $0 \le f(t) \le A^{-1}$ for all $t \in \mathbb{R}$. 
\item[(ii)] The support of $f$ is contained in $[B-2A \ell, B]$.
\item[(iii)]  The Laplace transform $F(z)=F_{\ell,A,B}(z)= \int_{\mathbb{R}}
f_{\ell,A,B}(t) e^{-zt}  dt$ is 
\begin{equation}
  F(z) = e^{-(B-2 A\ell)z} \Big(\frac{1-e^{-Az}  }{Az} \Big)^{2 \ell}. 
\end{equation}
\item[(iv)]  Let $\Lo \ge 1$ be arbitrary.  Suppose $s=\sigma+it \in \mathbb{C}$ satisfies
$\sigma < 1$ and $t \in \mathbb{R}$.  Write $\sigma=1-\frac{x}{\Lo}$ and $t=\frac{y}{\Lo}$.
If $0 \le \alpha \le 2 \ell$, then 
\begin{equation}
  \label{Fbd1}
  |F((1-s) \Lo)| \le e^{-(B-2 A\ell)x} \Big( \frac{2}{A \sqrt{x^2+y^2}  } \Big)^{\alpha}
  =e^{-(B-2 A\ell)(1-\sigma) \Lo} \Big( \frac{2}{A |1-s| \Lo  } \Big)^{\alpha}.
\end{equation}
Furthermore, 
\begin{equation}
  \label{Fbd2}
   |F((1-s) \Lo)| \le e^{-(B-2 A\ell)x} \text{ and } F(0)=1. 
\end{equation}
\end{itemize}
\end{lemma} 
In our argument, we shall show that \eqref{S} is positive for functions
 $f=f_{\ell,A,B}(t)$ for various choices of $\ell, A$, and $B$. 
In such cases, it follows that 
there is a prime $\mathfrak{p} \in \mathscr{P}$ so that $N \mathfrak{p}  \le d_{L}^{B}$.
We now summarise the arguments of \cite{LMO} and \cite{Z} that relate $S$ to the low lying zeros of $\zeta_L(s)$; we first consider the following function
\[
 \tilde{\Psi}_{\mathcal{C}}(s) =  
 \sum_{ 
  \mathfrak{p}\subseteq \mathcal{O}_K  } 
  \sum_{m\ge 1}
 \sum_{\sigma_{\mathfrak{p}}^m = \mathcal{C}} 
 \frac{\log N \mathfrak{p} }{\left(N\mathfrak{p}^m\right)^s} \text{ for $\Re (s)>1$}. 
\]
 Let $g_{\mathcal{C}}$ be a representative of $\mathcal{C}$ and $E$ be its fixed field.  Then Deuring's reduction \cite{De} ensures 
 that $ \tilde{\Psi}_{\mathcal{C}}(s)$  equals $\Psi_{\mathcal{C}}(s)$ plus an  error arising from the ramified primes where
\[
 \Psi_{\mathcal{C}}(s) = -\frac{|\mathcal{C}|}{|{\mathcal{G}|}} \sum_{\chi} \overline{\chi}(g_{\mathcal{C}}) \frac{L'}{L} (s,\chi,L/E),
\]
and the sum is over the irreducible characters of $\Gal(L/E)$. By \cite[Lemma 4.1]{Z}, 
$$
\Lo^{-1} S 
= \frac1{2\pi i} \int_{2-i\infty}^{2+i\infty} \Psi_{\mathcal{C}}(s) F((1-s)\Lo) ds +   \mathcal{E}_{1} ,
$$
where 
$\mathcal{E}_{1}$ arises from the contribution of the ramified prime ideals, the prime ideals $\mathfrak{p}$ with $N\mathfrak{p}$ non-rational prime, and the powers of prime ideals:
\begin{equation}\label{def-E2}
 \mathcal{E}_{1} \ll  A^{-1} \Lo^2 e^{-(B-2A\ell) \Lo/2}.
\end{equation}
Now as Artin reciprocity yields that $L (s,\chi,L/E)$ is a Hecke $L$-function of $E$, the classical analytic machinery can be invoked. 
We have from \cite[Lemma 4.2]{Z} that
\[
   \Big| \frac{|\mathcal{G}|}{|\mathcal{C}|} \Lo^{-1} S - F(0) \Big| \le 
   \sum_{\substack{\rho \\ |\gamma| \le T^{*}}} |F((1-\rho) \Lo)| 
   + \mathcal{E}_{1}    + \mathcal{E}_{2}   + \mathcal{E}_3   + \mathcal{E}_4,   
   \]
where 
$\mathcal{E}_{2}$ arises from the zeros with $|\gamma| > T^{*}$, 
$\mathcal{E}_3$ is from the zero at $s=0$, 
and $\mathcal{E}_4$ comes from the integral along the line $\Re (s) =-1/2$:
\begin{align}
& \label{def-E1} 
 \mathcal{E}_{2} \ll \Lo   \Big( \frac{2}{A T^{*} \Lo} \Big)^{2 \ell},
\\
& \label{def-E3} 
 \mathcal{E}_3 \ll  \Lo  \Big( \frac{1}{A \Lo } \Big)^{2 \ell} e^{-(B-2 A\ell) \Lo},
\\
& \label{def-E4} 
 \mathcal{E}_4 \ll  \Lo  \Big( \frac{2}{A \Lo } \Big)^{2 \ell} e^{-3(B-2 A\ell) \Lo/2}.
\end{align}
Finally, \cite[Lemma 4.3]{Z} removes the zeros outside of various zero-free regions. 
Let $J \ge 1$ be given and $T^{*} \ge 1$ be fixed.   Suppose 
\[
  1 \le R_1 \le R_2 \le \cdots \le R_J \le \Lo, \ 
  0 = T_0 < T_1 \le T_2 \le \cdots \le T_J= T^{*}. 
\]
We define $ \sum'$ as the sum over the zeros $\rho=\beta+i \gamma$
of $\zeta_L(s)$ satisfying 
\[
   \beta > 1-\frac{R_j}{\Lo}, \ T_{j-1} \le |\gamma| < T_j \text{ for some } 1 \le j \le J.
\]
Then
\begin{equation}
    \sum_{\substack{\rho \\ |\gamma| \le T^{*}}} |F((1-\rho) \Lo)|
    = 
    \sideset{}{'}\sum_{\rho}
    |F((1-\rho) \Lo)| + \mathcal{E}_5 + \mathcal{E}_6,
\end{equation}    
where 
$\mathcal{E}_5$ arises from the zeros in the region $\beta \le 1-\frac{R_1}{\Lo}, |\gamma| < T_1$, and 
$\mathcal{E}_6$ is from the zeros in the union over $j=2,\ldots,J$ of the regions $\beta \le 1-\frac{R_j}{\Lo}, T_{j-1} \le |\gamma| < T_j$:
\begin{align}
& \label{def-E5}
\mathcal{E}_5 \ll \min\Big(  \Big( \frac{2}{A} \Big)^{2 \ell},  \Lo \Big)e^{-(B-2 A\ell) R_1},
\\
& \label{def-E6}
\mathcal{E}_6 \ll \sum_{j=2}^{J} \Lo \Big( \frac{2}{A T_{j-1} \Lo} \Big)^{2 \ell} e^{-(B-2 A\ell)R_j}.
\end{align}
Note that if $J=1$, then $\mathcal{E}_6$ vanishes. 
Recognizing $F(0)=1$, we summarise all of the above in the following lemma:
\begin{lemma} \label{lemmaTj}
Let $J \ge 1$ be given and $T^{*} \ge 1$ be fixed.   Suppose 
\[
  1 \le R_1 \le R_2 \le \cdots \le R_J \le \Lo, \ 
  0 = T_0 < T_1 \le T_2 \le \cdots \le T_J= T^{*}. 
\]
Then
\begin{equation}
  \frac{|\mathcal{G}|}{|\mathcal{C}|} \Lo^{-1} S\ge 1 - 
 \sideset{}{'}\sum_{\rho} |F((1-\rho) \Lo)| 
     + \mathcal{O}\left(  \mathcal{E}_{1}    + \mathcal{E}_{2}   + \mathcal{E}_3   + \mathcal{E}_4   + \mathcal{E}_5 + \mathcal{E}_6\right),
\end{equation}    
where the $\mathcal{E}_i$'s are defined in \eqref{def-E2}, \eqref{def-E1}, \eqref{def-E3}, \eqref{def-E4}, \eqref{def-E5}, and \eqref{def-E6}. 
\end{lemma}
This is the key formula allowing us to show that $S$ is positive in various cases.  There are four cases to consider. 
Let $\tilde{\eta}\in (0,\frac{1}{12.74})$
be a sufficiently small  absolute positive  constant.
\begin{enumerate}
\item[1.] {\bf The non-exceptional case:} \\
Assume all zeros $\beta+i\gamma$ satisfy $\beta=1-\frac{\lambda}{\Lo}$
with 
$
 \lambda\ge \tfrac1{12.74} .
$
\item[2.] {\bf The exceptional cases:}\\
We assume there is an exceptional zero $\beta_1 =1-\frac{\lambda_1}{\Lo}$ with 
$ \lambda_1 < \frac1{ 12.74} .$
\begin{enumerate}
\item[(i)] {\it $\lambda_1$ {\it small}:}
$\displaystyle{ \tilde{\eta}\le \lambda_1 < \tfrac1{12.74} }$ (with $\tilde{\eta}$ fixed parameter $0<\tilde{\eta}<1$).
\item[(ii)] {\it $\lambda_1$ {\it very small}:}
$\displaystyle{  \Lo^{-200} \le \lambda_1 < \tilde{\eta} }$.
\item[(iii)] {\it $\lambda_1$ {\it extremely small}}:
$\displaystyle{  \lambda_1 < \Lo^{-200} }$.
\end{enumerate}
\end{enumerate}
We shall not consider Cases 1 and 2.(i); for the proofs we refer to  \cite[pp.136-138]{Z}.  
Note that the proofs in these cases make use of a number of results from \cite{K} and \cite{KN}.
Namely, the zero-free region for $\zeta_{L}(s)$ \eqref{zfr} 
\cite[Theorem 1.1]{K},  the Deuring-Heilbronn phenomenon for zeros close to $s=1$
\cite[Theorem 4]{KN},  and a smoothed explicit local formula for zeros of the Dedekind zeta function
\cite[Lemma 7]{KN}. 
Instead, we shall improve Cases 2.(ii) and (iii).
In the following table, in column 2,  we record  Zaman's  values of $B$ for which it can be shown that
$S$ is positive for $f=f_{\ell,A,B}(t)$ (for some $\ell$ and $A$),  in each case.   Column 3 lists our improvements to Zaman's values.
 \begin{center}
  \begin{tabular}{|l|c|c|  }
  \hline
  Cases & $B$ (Zaman)  &   $B$  \\
 \hline
  \hline
  1. The non-exceptional case       & $7.41$ &  $-$  \\
 \hline
 2.(i) $\lambda_1$  {\it small}           & $2.63$ & $-$   \\
 \hline
 2.(ii)  $\lambda_1$ {\it very small}      & $36.5$ &  12.48  \\
 \hline
 2.(iii)  $\lambda_1$ {\it extremely small}  & $39.5$ &  15.72  \\
 \hline
 \end{tabular}
 \end{center}
 This table shows that, in each case, there exists a prime ideal $\mathfrak{p} \in  \mathscr{P} $ such that
 $
 N \mathfrak{p} \le d_{L}^{B} 
 $ where $B$ is the listed constant(s).  It follows that, we can unconditionally choose $B=15.72$ and this establishes 
 Theorem \ref{mainthm}.  Note that we do not attempt to improve cases 1 and 2.(i) as we currently are unable to reduce the constants
 in Cases 2.(ii) and (iii) below 7.41.  We now treat the last two cases.
\subsection{Very small case:  $\Lo^{-200} \le \lambda_1 < \tilde{\eta}$} We shall select the weight function 
$f=f_{\ell,A,B}$ with parameters $\ell=101$ and $A$ and $B$ to be determined. 
Note that for $\eta=0.1$, Theorem \ref{DH_T} holds with $C=12.262$ (see \eqref{lowzfr}). 
For now we assume that $B-2 A\ell > C= 12.262$. 
%
The selection of these choices shall be explained shortly. 
By \eqref{def-E2}, \eqref{def-E1}, \eqref{def-E3}, and \eqref{def-E4}, 
\begin{align*}
& \mathcal{E}_{1} \ll  A^{-1} \Lo^2 e^{-(B-2A\ell) \Lo/2}  \ll \Lo^{-2\ell+1},\enspace 
 \mathcal{E}_{2} \ll \Lo   \Big( \frac{2}{A T^{*} \Lo} \Big)^{2 \ell} \ll \Lo^{-2\ell+1}, 
\\&  \mathcal{E}_3 \ll  \Lo  \Big( \frac{1}{A \Lo } \Big)^{2 \ell} e^{-(B-2 A\ell) \Lo}  \ll \Lo^{-2\ell+1}, 
\\&   \mathcal{E}_4 \ll  \Lo  \Big( \frac{2}{A \Lo } \Big)^{2 \ell} e^{-3(B-2 A\ell) \Lo/2}  \ll \Lo^{-2\ell+1}.
\end{align*}
We apply Lemma \ref{lemmaTj} with $J=1, T^{*}=1, $ and $R_1 = \frac{\log ( \kappa \lambda_1^{-1} )}{C}$, for some $\kappa >0$,  to conclude that 
there is no error term $\mathcal{E}_6$, and that
\[
\mathcal{E}_5 \ll \min\Big(  \Big( \frac{2}{A} \Big)^{2 \ell},  \Lo \Big)e^{-(B-2 A\ell) R_1} 
\ll  e^{-(B-2 A\ell) R_1} 
\ll \lambda_1^{\frac{B-2A\ell }{C}}.
\]
Note that $\sum'$  is over the zeros $\rho=\beta +i\gamma$ with $\beta > 1-\frac{R_1}{ \Lo}$ and $|\gamma|\le 1$, 
and that for any fixed $\eta\in(0,1)$, as $ \Lo^{-2 \ell+2} \le \lambda_1$, we have
$R_1 \le  \frac{1}{12.26}(\log \kappa +(2 \ell-2)\log \Lo) = o(\Lo)$.  
Thus, for $\Lo$ sufficiently large, 
$1-\eta< 1- \frac{R_1 }{\Lo}$.
From this, Theorem \ref{DH_T} ensures that any other zero $\rho=\beta+i\gamma\neq\beta_1$, with $|\gamma|\le 1$, has to satisfy $\beta \le  1-\frac{R_1 }{\Lo}$, and hence $\sum'$ only contains the exceptional zero $\beta_1$. Thus,
\[
    \frac{|\mathcal{G}|}{|\mathcal{C}|} \Lo^{-1} S \ge 1-|F((1-\beta_1)\Lo)| 
+ \mathcal{O}\left(  \Lo^{-2\ell+1}\right)  + \mathcal{O}\left( \lambda_1^{\frac{B-2A\ell }{C}}\right). 
\]
Note that by \eqref{Fbd2}, we have $|F((1-\beta_1)\Lo)| \le  e^{-(B-2A\ell) \lambda_1}$.
Thus,
\[
   \frac{|\mathcal{G}|}{|\mathcal{C}|}\Lo^{-1} S 
   \ge 1 - e^{-(B-2A\ell) \lambda_1}
+ \mathcal{O}\left(  \Lo^{-2\ell+1}\right)  + \mathcal{O}\left( \lambda_1^{\frac{B-2A\ell }{C}}\right).
\]
Using 
$1-e^{-x} > x -\frac{x^2}2 \ge \frac{3x}{4}$ for $0 < x \le \frac{1}{2}$, the above becomes positive whenever
\[
 \tfrac{3}{4}(B-2A\ell) \lambda_1
+ \mathcal{O}\left(  \Lo^{-2\ell+1}\right)  + \mathcal{O}\left( \lambda_1^{\frac{B-2A\ell }{C}}\right) >0,
\]
which happens if
\[ 
(B-2A\ell) \lambda_1 \gg   \Lo^{-2\ell+1}
\ \text{and}\ 
(B-2A\ell) \lambda_1 \gg \lambda_1^{\frac{B-2A\ell }{C}}. 
\]
The first condition is true under our assumption $\Lo^{-200}= \Lo^{-2\ell+2 } \le \lambda_1$.
Since $\lambda_1<1$, the second condition is true if 
 $\frac{B-2A\ell }{C}>1$.  Thus, we must choose $B$ as small as possible such that $B> C+2A\ell=12.262+2 A \ell$.
 By a computer calculation, we obtain 
 \begin{equation}
  \label{parameters1}
   \ell = 101 , \ B = 12.48 , \ A= \frac{1}{970}, \ \text{ and } B-2 A\ell = 12.271\ldots  .
\end{equation}
 \subsection{Extremely small case: $\lambda_1 < \Lo^{-200}$}  
 We shall assume that there are positive increasing parameters $T_j$ for $0 \le j \le J$
 for some $J \ge 1$.
 Associated to each $j$ are positive constants $C_j$ and $\kappa_j$ such that 
 if $\beta'+i \gamma'$ is another zero of $\zeta_L(s)$ with $\beta' \ge \frac{1}{2}$ and $|\gamma'| \le T_j$, then 
 \begin{equation}\label{bdd_beta'2}
       \beta' \le  1 - \frac{ \log \Big( \frac{\kappa_j}{(1-\beta_1) \Lo} \Big)}{C_j \Lo}.
 \end{equation}
The parameters $C_j$ shall increase with $j$ and we set $T_0=1$, $C_0 =14.144$. 
Note that 
\begin{equation}
  \label{Lobd}
\Lo < \lambda_1^{-1/{200}},
\end{equation}
and from Corollary \ref{upperbd}, 
$1-\beta_1 \gg d_{L}^{-\frac{C_0}2}$
where $C_0 =14.144$,
which implies that 
\begin{equation}
\label{rng-lambda1}
\Lo e^{-\frac{C_0}2\Lo} \ll \lambda_1 < \Lo^{ -200}.
\end{equation}
We define $0<u<1<v$ to be positive parameters. 
We choose $\ell = \ceil{ v \Lo } \in [v\Lo, v\Lo+1), A = \frac{u}{\Lo}$, so that $A\Lo = u \ll 1$ and $ uv\le A\ell < 
uv+ O(\Lo^{-1})=uv+o(1)$. Throughout this section, we write   $o(1)$ to denote a term which approaches $0$ as $d_L \to \infty$. 
By \eqref{def-E2}, \eqref{def-E1}, \eqref{def-E3}, and \eqref{def-E4}, we find 
\begin{align*}
& \mathcal{E}_{1} \ll  A^{-1} \Lo^2 e^{-(B-2A\ell) \Lo/2}  \ll  \Lo^3 e^{-\frac{B-2uv+o(1)}{2} \Lo}  ,
 \\ & \mathcal{E}_{2} \ll \Lo   \Big( \frac{2}{A T^{*} \Lo} \Big)^{2 \ell} \ll \Lo   e^{-  2 v \log(\frac{u T^{*}}{2}) \Lo }, 
\\&  \mathcal{E}_3 \ll  \Lo \Big( \frac{1}{A \Lo } \Big)^{2 \ell} e^{-(B-2 A\ell) \Lo}  \ll \Lo    e^{-(B-2  uv-2 v \log  (1/u)+o(1)) \Lo},
\\&   \mathcal{E}_4 \ll  \Lo  \Big( \frac{2}{A \Lo } \Big)^{2 \ell} e^{-3(B-2 A\ell) \Lo/2}  \ll \Lo    e^{-(\frac{3B}{2}-3  uv-2 v \log  (2/u)+o(1)) \Lo} , 
\end{align*}
and thus
\begin{equation}\label{bnd-E1234}
\mathcal{E}_{1}+\mathcal{E}_{2}+\mathcal{E}_3+\mathcal{E}_4 
\ll \Lo^3 e^{-\frac{B-2uv+o(1)}{2} \Lo}  + \Lo   e^{-  2 v \log(\frac{u T^{*}}{2}) \Lo }. 
\end{equation}
Using \eqref{Fbd2} and \eqref{rng-lambda1}, and assuming $B-2uv>0$, we have
\begin{equation}
1-|F((1-\beta_1)\Lo)| > \tfrac{3}{4}(B-2uv+o(1)) \lambda_1 \gg (B-2uv+o(1))\Lo e^{-\frac{C_0}2\Lo}  .
\end{equation}
We assume 
$\frac{B-2uv}{2}-\frac{C_0}2>0$ 
and 
$ \frac{C_0}2 -  2 v \log(\frac{u T^{*}}{2})<0$, so that
\[
\Lo^3 e^{-\frac{B-2uv+o(1)}{2} \Lo}  = \Lo e^{-\frac{C_0}2\Lo} e^{- (\frac{B-2uv+o(1)}{2}-\frac{C_0}2-\frac{2\log\Lo}{\Lo}) \Lo } =o \Big(\Lo e^{-\frac{C_0}2\Lo} \Big),\text{ and}
\]
\[
\Lo   e^{-  2 v \log(\frac{u T^{*}}{2}) \Lo } = \Lo   e^{-\frac{C_0}2\Lo} e^{(\frac{C_0}2 -  2 v \log(\frac{u T^{*}}{2}) )\Lo }  =o \Big(\Lo e^{-\frac{C_0}2\Lo}\Big) .
\]
By choosing $B$ and $T^{*}$ with $B> C_0 + 2uv$ and $ T^{*}  > \frac{2}{u}e^{ \frac{C_0}{4v} }$, we then have 
\[ \mathcal{E}_{1}+\mathcal{E}_{2}+\mathcal{E}_3+\mathcal{E}_4 = o(1-|F((1-\beta_1)\Lo)| ) .\]
We now apply Lemma \ref{lemmaTj} with $J \ge 1$, $T^{*}=T_{J} >   \frac{2}{u}e^{ \frac{C_0}{4v} }$, and $R_j = \frac{\log ( \kappa_{j}\lambda_1^{-1} )}{C_j}$ for each $j\in\{1,\ldots,J\}$. Now by
\eqref{bdd_beta'2}, for any non-trivial zero $\rho=\beta+i\gamma\neq\beta_1$ with $|\gamma|\le T_j$, either $\beta\le 1-\frac{1}{2}$ or $\beta\le 1-\frac{R_j }{\Lo}$. From the symmetry of the non-trivial zeros of $\zeta_L(s)$ with respect to $\Re(s)=\frac{1}{2}$, it follows that  
$\beta \le 1-\frac{R_j }{\Lo}$, and thus the sum $\sum'$ contains at most $\beta_1$ and $1-\beta_1$, where the latter only contributes a negligible error.
%
By our choice $R_1 = \frac{\log( \kappa_{1}\lambda_1^{-1})}{C_1}$ and \eqref{Lobd}, we have
\[
\mathcal{E}_5 
\ll  \Lo  e^{-(B-2 A\ell) R_1} 
\ll   \Lo  \left(\kappa_1^{-1}\lambda_1\right)^{ \frac{(B-2 uv) }{C_1} } \ll   \lambda_1^{ \frac{(B-2 uv) }{C_1} -\frac1{200}} = o(\lambda_1)
\]
if we assume $B> \left(1+\frac1{200}\right)C_1+2uv$. Finally, we analyse \eqref{def-E6}
\[
\mathcal{E}_6
 \ll \sum_{j=2}^{J} \Lo \Big( \frac{2}{A T_{j-1} \Lo} \Big)^{2 \ell} e^{-(B-2 A\ell)R_j}
 \ll \Lo \sum_{j=2}^{J} e^{-2 v\Lo\log \big( \frac{uT_{j-1}}{2} \big) }  \lambda_1^{ \frac{ B-2 uv  }{C_j} }  .
\]
Note that choosing 
$
    \frac{B-2uv}{C_j} \ge 1 \text{ for }  2 \le j \le J 
$
would immediately give $\mathcal{E}_6 =o(\lambda_1)$
as long as $T_1 > \frac{2}{u}$. 
However, this would force us to take $B\ge C_J+2uv$.
As we are trying to minimize $B$, we instead  assume  
$
    \frac{B-2uv}{C_j} < 1 \text{ for }  2 \le j \le J.
$
This implies $2uv<B < C_2+2uv$; we also have
\begin{align*}
\mathcal{E}_6
 & \ll \lambda_1
  \Lo \sum_{j=2}^{J} e^{-2 v\Lo\log \big( \frac{uT_{j-1}}{2} \big) }  \Big(\Lo e^{-\frac{C_0}2 \Lo}\Big)^{ \frac{ B-2 uv  }{C_j} -1}  
  \ \text{ (since $\lambda_1 \gg \Lo e^{-\frac{C_0}2 \Lo}$)}
\\& \ll \lambda_1
\sum_{j=2}^{J} 
\Lo^{ \frac{ B-2 uv  }{C_j} }
e^{-\left( 2 v \log \big( \frac{uT_{j-1}}{2} \big) - \frac{C_0}2 \left(1-\frac{ B-2 uv  }{C_j} \right) \right)\Lo}.
\end{align*}
It follows that $E_6 = o(\lambda_1)$ if $B   > \big( 1-\tfrac{4 v}{C_0} \log \big( \tfrac{uT_{j-1}}{2} \big) \big) C_j+2 uv$
for all  $2 \le j \le J$. 
To conclude, the desired inequality $ \frac{|\mathcal{G}|}{|\mathcal{C}|}\Lo^{-1} S >0 $ follows from the assumptions
\begin{equation}\label{firstassumptions}
 T^{*} =T_{J} > \frac{2}{u}e^{ \frac{C_0}{4v} },  \  (1+\tfrac1{200} )C_1   + 2uv
< B, \text{ and }
\end{equation}
\begin{equation}\label{conditionsBT}
\max_{2 \le j \le J} \big(  \big( 1-\tfrac{4 v}{C_0} \log \big( \tfrac{uT_{j-1}}{2} \big) \big) C_j\big) + 2uv
< B < C_2+2uv.
\end{equation}
(Note that the condition $C_0 +2uv < B$ is dropped since $C_0 < C_1$.)
We employ a computer search to determine admissible parameters $T_j$, $C_j$, $B$, $u$, and $v$.
We now apply Theorem \ref{DH_T} to obtain $C_j=C(c,\tfrac{1}{2},T_j,\varepsilon)$ for numerically optimal choices of $c$ and $\varepsilon$.  We use the parameters
\begin{equation}
  T_1 = 4.6, \ 
  C_1 = 14.58 \ldots, \ 
  T_2 = 10, \ 
  C_2 = 15.50 \ldots, \
  T_3 = 130, \ , C_3 = 20.21 \ldots. 
\end{equation}
(see \eqref{highzfr} just after Theorem \ref{DH_T}) and 
\begin{equation}
  B =15.72, \
  u = 0.53, \ v= 1.001,  \ A = u \Lo^{-1}, \  \ell = \ceil{ v \Lo }, \  J = 3.
\end{equation}
These satisfy the conditions \eqref{firstassumptions} and \eqref{conditionsBT} and thus complete our proof. 

\end{document}